\documentclass[10pt,a4paper,reqno]{article}
\usepackage{amssymb,amsthm,enumerate,color}
\usepackage{newtxtext,newtxmath}
\usepackage{amsmath,fancybox,amsfonts,bm,fullpage}  

\usepackage{tikz}
\usetikzlibrary{intersections, calc}
\usetikzlibrary{decorations.markings}
\tikzset{
  arrows along my path/.style={
    postaction={
      decorate,
      decoration={
        markings,
        mark=between positions 0.03 and 1 step 10pt with {\arrow{Stealth[length=5pt]}},
   }}}}
\providecommand{\U}[1]{\protect\rule{.1in}{.1in}}

\theoremstyle{plain}
\newtheorem{thm}{Theorem}[section]

\newtheorem{cor}[thm]{Corollary}

\newtheorem{lem}[thm]{Lemma}

\newtheorem{prop}[thm]{Proposition}

\newtheorem{con}[thm]{Conjecture}

\theoremstyle{definition}
\newtheorem{defn}{Definition}
\theoremstyle{remark}
\newtheorem{rem}[thm]{Remark}

\newcommand{\C}{\mathbb{C}}
\newcommand{\R}{\mathbb{R}}

\newcommand{\N}{\mathbb{N}}

\newcommand{\calB}{\mathcal{B}}

\newcommand{\calP}{\mathcal{P}}
\newcommand{\calR}{\mathcal{R}}

\newcommand{\rmd}{\mathrm{d}}

\renewcommand{\Re}{{\sf Re}}
\renewcommand{\Im}{{\sf Im}}

\begin{document}
\title{Free self-decomposability and unimodality of \\ the Fuss-Catalan distributions}
\author{Wojciech M{\l}otkowski\thanks{W.~M. is supported by the Polish
National Science Center grant No. 2016/21/B/ST1/00628.} 
\and Noriyoshi Sakuma\thanks{N.~S is supported by JSPS KAKENHI Grant Numbers15K04923, 19K03515.}
\and Yuki Ueda}
\date{\today}
\maketitle

\begin{abstract}
We study properties of the Fuss-Catalan distributions $\mu(p,r)$,
$p\geq1$, $0<r\leq p$: free infinite divisibility,
free self-decomposability, free regularity and unimodality.
We show that the Fuss-Catalan distribution $\mu(p,r)$ is freely
self-decomposable if and only if  $1 \leq p=r \leq 2$.
\end{abstract}


\section{Introduction}

The \textit{two-parameters Fuss-Catalan numbers} (called also \textit{Raney numbers}) are defined by
\[
A_{k}(p,r):= \frac{r}{k!} \prod_{i=1}^{k-1} (kp +r-i) = \frac{r}{kp+r} \binom{kp+r}{k},
\]
where $p,r$ are real parameters. If $p,r$ are natural numbers then $A_k(p,r)$ may
admit several combinatorial interpretations in terms of plane trees, lattice paths
or noncrossing partitions, see~\cite{concretemathematics}.
In particular, $A_k(2,1)$ is the famous Catalan sequence, see~\cite{stanley2015}.

It is known that the sequence $A_{k}(p,r)$ is positive definite
if and only if either $p\geq1$, $0<r\leq p$ or $p\leq 0$, $p-1\geq r$ or else if $r=0$,
see \cite{Mlot,MP2014,MPZ2013,FL2015,liupego2014} for various proofs.
The corresponding probability measure we will call the \textit{Fuss-Catalan distribution}
and denote $\mu(p,r)$, so that
\[
A_{k}(p,r) =\int_{\R} x^{k}\mu(p,r)(\rmd x).
\]
It is easy to observe that $\mu(p,0)=\delta_{0}$ and that
$\mu(1-p,-r)$ is the reflection of $\mu(p,r)$.
Therefore we will confine ourselves to the case $p\geq1$, $0<r\leq p$.
Then $\mu(p,r)$ is absolutely continuous
(except $\mu(1,1)=\delta_1$) and the support is
$\left[0,p^{p}(p-1)^{1-p}\right]$
for $p>1$ and $[0,1]$ for $p=1$, $0<r<1$.
Forrester and Liu \cite{FL2015} found the following implicit formula for the density function $W_{p,r}(x)$
of $\mu(p,r)$:

\begin{prop}[Proposition 2.1 in \cite{FL2015}]\label{Prop:density}
For $p>1$ put
\begin{equation}\label{eq:propdensityrho}
\rho(\varphi):= \frac{(\sin (p \varphi))^{p}}{\sin(\varphi) (\sin((p-1)\varphi))^{p-1}}, \quad 0<\varphi<\frac{\pi}{p}.
\end{equation}
Then $\rho(\varphi)$ is a decreasing function which maps $(0,\pi/p)$ onto $\left(0,p^{p}(p-1)^{1-p}\right)$
and we have
\begin{equation}\label{eq:propdensitywpr}
W_{p,r}(\rho(\varphi)) =
\frac{(\sin((p-1)\varphi))^{p-r-1}\sin(\varphi)\sin(r \varphi)}{\pi (\sin (p \varphi))^{p-r}} , \quad 0<\varphi<\frac{\pi}{p}.
\end{equation}
\end{prop}

For $p=1$, $0<r<1$ the density function $W_{1,r}(x)$ is given in \cite[formula (5.2)]{Mlot} (or see Proposition \ref{prop:W_1,r}).

We have therefore
\begin{equation}\label{eg:akprintegral}
A_k(p,r)=\int_{0}^{p^p(p-1)^{1-p}} x^k W_{p,r}(x)\,\rmd x
\end{equation}
for $k\ge0$, $p>1$, $0<r\leq p$. This formula is still valid for $r>p>1$, however in this case $W_{p,r}(x)$
is negative on some subinterval of $\left(0,{p^p(p-1)^{1-p}}\right)$.
Another description of the density function $W_{p,r}$, in terms of the Meijer functions,
for rational $p>1$, was provided in \cite{MPZ2013}.

It was proved in \cite{Mlot} that the free cumulant sequence $\{r_{k}(\mu(p,r))\}_{k=1}^{\infty}$
of $\mu(p,r)$ is given by
\begin{equation}\label{freecumulants of mu(p,r)}
r_{k}(\mu(p,r))=A_{k}(p-r,r).
\end{equation}
Consequently, if $0<r\leq\min\{p-1,p/2\}$ then $\mu(p,r)$ is freely infinitely divisible,
i.e infinitely divisible with respect to the additive free convolution $\boxplus$.
Here we will show in addition that $\mu(p,p)$ is freely infinitely divisible
for $1\leq p\leq2$.

\subsection{Main results}
In this paper we study the Fuss-Catalan distributions $\mu(p,p)$ and $\mu(p,r)$ in the framework of the free probability theory.
In particular we are interested in free infinite divisibility, free self-decomposability, free $L_1$ property,
free regularity and also unimodality.
First we briefly recall these concepts.
In Section~3 we concentrate on the distributions $\mu(p,p)$. We will prove:

\begin{thm} For the Fuss-Catalan distribution $\mu(p,p)$ we have the following:
\begin{enumerate}
\item $\mu(p,p)$ is freely infinitely divisible if and only if $1\leq p \leq 2$.
\item If $1\leq p \leq 2$ then $\mu(p,p)$ is freely self-decomposable, more precisely, it is in the free $L_1$ class.
\item $\mu(p,p)$ is not free regular for any $1<p<2$.
\item $\mu(p,p)$ is unimodal for all $p\geq 1$.
\end{enumerate}
\end{thm}
In Section 4 we obtain some results Fuss-Catalan distribution $\mu(p,r)$ in general:

\begin{thm}\label{thm:main}
Suppose that $p\geq 1$ and $0<r\leq p$. For the Fuss-Catalan distribution $\mu(p,r)$ we have the following:
\begin{enumerate}
\item $\mu(p,r)$ is freely infinitely divisible if and only if either $0<r\leq \min\{p/2,p-1\}$ or $1\leq p=r \leq 2$.
\item $\mu(p,r)$ is freely self-decomposable if and only if $1\leq p=r \leq 2$.
\item $\mu(p,r)$ is free regular if and only if either $0<r\leq \min\{p/2,p-1\}$ or $p=r\in\{1,2\}$.
\end{enumerate}
\end{thm}

Furthermore, we study the unimodality for the Fuss-Catalan distributions $\mu(p,p-1)$, $\mu(2r,r)$, $\mu(1,r)$ and $\mu(2,r)$.

In this paper we will denote by $\calP(I)$ and $\calB(I)$ the family of all Borel probability measures
and the class of all Borel sets on $I\subseteq\mathbb{R}$.
We will denote $\C^+$ (resp.\ $\C^-$) the set of complex numbers with strictly positive (resp.\ strictly negative) imaginary part.

\section{Preliminaries in the free probability theory}
\subsection{Freely infinitely divisible distributions}

The notion of the free infinite divisibility is an important research area.
One reason is the Berovici-Pata map, which is a bijection
which maps classical infinitely divisible distributions onto
free infinitely divisible ones.

A probability measure $\mu$ on $\R$ is called \emph{freely infinitely divisible} if for any $n\in\N$ there exists a probability measure $\mu_{n}\in\calP(\R)$ such that
\begin{align*}
\mu = \underbrace{\mu_{n}\boxplus \dots \boxplus \mu_{n}}_{\text{$n$ times}},
\end{align*}
where $\boxplus$ denotes the {\it free additive convolution} which can be defined
as the distribution of sum of freely independent selfadjoint operators.
In this case, $\mu_n\in\calP(\R)$ is uniquely determined for each $n\in\N$.
The freely infinite divisible distributions
can be characterized as those admitting a L{\'e}vy-Khintchine representation in terms of $R$-transform which is the free analog of the cumulant transform $C_{\mu}(z) := \log (\widehat{\mu}(z))$, where $\widehat{\mu}$ is the characteristic function of $\mu$.
This was originally established by
Bercovici and Voiculescu in \cite{BeVo1993}.
To explain it, we gather analytic tools for free additive convolution $\boxplus$. On order to define
the \emph{$R$-transform} (or \emph{free cumulant transform}) $R_\mu$ of a (Borel-) probability measure
$\mu$ on $\R$ first we need to define its
Cauchy-Stieltjes transform $G_\mu$:
\begin{equation*}
G_\mu(z)=\int_{\R}\frac{1}{z-t}\,\mu(\rmd t), \qquad(z\in\C^+).
\end{equation*}
Note in particular that $\Im(G_\mu(z))<0$ for any $z$ in $\C^+$, and
hence we may consider the reciprocal Cauchy transform
$F_\mu\colon\C^{+}\to\C^{+}$ given by $F_{\mu}(z)=1/G_{\mu}(z)$.
For any probability measure $\mu$ on $\R$ and any $\lambda$ in
$(0,\infty)$ there exist
positive numbers $\alpha,\beta$ and $M$ such that $F_{\mu}$ is univalent
on the set $\Gamma_{\alpha,\beta}:=\{z \in \C^{+} \,|\, \Im(z) >\beta,
|\Re(z)|<\alpha \Im(z)\}$ and such that
$F_{\mu}(\Gamma_{\alpha,\beta})\supset\Gamma_{\lambda,M}$.
Therefore the right inverse $F^{-1}_{\mu}$ of $F_{\mu}$ exists on
$\Gamma_{\lambda,M}$, and the free cumulant transform
$ R_\mu$ is defined by
\begin{align*}
 R_{\mu}(w) =wF^{-1}_{\mu}(1/w)-1, \quad\text{for all $w$ such that
  $1/w \in \Gamma_{\lambda,M}$}.
\label{def_Cmu_eq}
\end{align*}
The name refers to the fact that $ R_\mu$ linearizes free additive
convolution (cf.\ \cite{BeVo1993}). Variants of $ R_\mu$ (with the
same linearizing property) are the $R$-transform $\calR_\mu$ and the
Voiculescu transform $\varphi_\mu$ related by the following equalities:
\begin{equation*}
 R_\mu(w)=w\calR_\mu(w)=w\varphi_\mu(\tfrac{1}{w}).
\label{relations_eq}
\end{equation*}

The free version of the L\'evy-Khintchine representation now amounts
to the statement that a probability measure $\mu$ on $\R$
is freely infinitely divisible if and only if there exist $a\ge 0$,
$\eta\in\R$ and a L{\'e}vy measure\footnote{A (Borel-) measure $\nu$
  on $\R$ is called a L\'evy measure, if $\nu(\{0\})=0$ and
  $\int_{\R}\min\{1,x^2\}\,\nu(\rmd x)<\infty$.}
$\nu$ such that
\begin{align}
 R_{\mu}(w) = a w^{2}+\eta w +
\int_{\R}\left(\frac{1}{1- w x}-1-w x \mathbf{1}_{[-1,1]}(x)\right)\nu(\rmd x)\qquad(w\in\C^-).
\label{eqno1}
\end{align}
The triplet $(a,\eta,\nu)$ is uniquely determined and referred to as
the \emph{free characteristic triplet} for $\mu$, and $\nu$ is referred to as the \emph{free L\'evy measure} for $\mu$. In terms of the Voiculescu
transform $\varphi_\mu$ the free L\'evy-Khintchine representation takes
the form:
\begin{equation}
\label{eqno1a}
\varphi_{\mu}(z)=\gamma+\int_{{\mathbb R}}\frac{1+tz}{z-t}\, \sigma(\rmd t),
\qquad (z\in{\mathbb C}^+),
\end{equation}
where the \emph{free generating pair} $(\gamma,\sigma)$ is uniquely
determined and related to the free characteristic triplet by the
formulas\cite{BNT06}:
\begin{align*}
\begin{cases}
\nu({\rm d}t)=\frac{1+t^2}{t^2}\cdot \mathbf{1}_{{\mathbb R}\setminus\{0\}}(t) \
\sigma({\rm d}t),\\
\eta=\gamma+\int_{{\mathbb R}}t\Big(\mathbf{1}_{[-1,1]}(t)-\frac{1}{1+t^2}\Big) \
\nu({\rm d}t), \\
a=\sigma(\{0\}).
\end{cases}
\end{align*}
In particular $\sigma$ is a finite measure.
The right hand side of \eqref{eqno1a} gives rise to an analytic
function defined on all of $\C^+$, and in fact the property that
$\varphi_\mu$ can be extended analytically to all of $\C^+$ also
characterizes the measures in the class of freely infinitely divisible distributions. More precisely Bercovici and Voiculescu proved in \cite{BeVo1993} the following
fundamental result:

\begin{thm}
A probability measure $\mu$ on $\R$ is freely infinitely divisible if and only if
the Voiculescu transform $\varphi_{\mu}$ has an analytic extension
defined on $\C^{+}$ with values in $\C^{-}\cup \R$.
\label{BV_char_of_FID}
\end{thm}
Recently free infinite divisibility has been proved for: normal distribution, some of the Boolean-stable distributions, some of the beta distributions and some of the gamma distributions, including the chi-square distribution, see \cite{AH14,BBLS11, H14}.

\subsection{Freely self-decomposable distributions}
A probability measure $\mu$ on $\R$ is called \emph{freely self-decomposable} if for any $c\in(0,1)$ there exists a probability measure $\rho_{c}\in\calP(\R)$ such that
\begin{equation}\label{self-decomposabledefinition}
\mu = D_{c}(\mu)\boxplus \rho_{c},
\end{equation}
where $D_{c}$ is dilation, that is, $D_{c}(\mu)(B) := \mu(c^{-1}B)$ for any $c>0$ and $B\in\calB(\R)$.
It is known that the measures $\rho_{c}$ are freely infinitely divisible too.

The concept has six characterizations at least:
(I) in terms of the free L{\'e}vy measure, (II) limit theorem,
{(III) stochastic integral representation}, {(IV) self-similarity}, {(V) the free cumulant sequence} and {(VI) analytic functions}.
From (I) to (IV) are proved in \cite{BNT06} based on the Bercovici-Pata bijection and classical results (see page 2 in \cite{Sato2010}).
(V) and (VI) are proved in  \cite{HTS}.
In this paper we will apply {(I)} and {(V)}.

\subsubsection{Free L{\'e}vy measure}
In this section we characterize the class of classical and freely self-decomposable distributions via its L\'{e}vy measure. Firstly we define the concept of unimodality. A measure $\rho$ is said to be \emph{unimodal with mode $a\in\mathbb{R}$} if
\begin{align*}
\rho(dx)=c\delta_a+ f(x)\,\rmd x,
\end{align*}
where $\rmd x$ is the Lebesgue measure, $c\geq0$ and $f(x)$ is a function which is non-decreasing on $(-\infty,a)$ and non-increasing on $(a,\infty)$.
If $\mu$ is classically or freely self-decomposable, then its corresponding L{\'e}vy measure
is absolutely continuous with respect to Lebesgue measure and classical/free L{\'e}vy measure
$\nu_{\mu}$ has following form:
\begin{align*}
\nu_\mu(\rmd x)=\frac{k(x)}{|x|}\,\rmd x,
\end{align*}
where the measure $k(x)\,\rmd x$ is unimodal with mode $0$, see \cite{BNT06} for details.



\subsubsection{Free cummulant sequence}

If $\mu$ is a compactly supported probability measure on $\R$ then the free cumulant transform
$R_\mu$ can be extended analytically to an
open neighborhood of $0$ and $R_{\mu}(0)=0$. Thus $R_\mu(z)$ admits a power series
expansion:
\begin{equation*}
 R_\mu(z)=\sum_{n=1}^{\infty} r_{n}(\mu)z^n
\end{equation*}
in a disc around $0$. The coefficients $\{ r_n(\mu)\}_{n\ge1}$
are called the \textit{free cumulants of $\mu$} (see e.g.\ \cite{BG2006}).
The can be also computed from moments of $\mu$ via M\"obius inversion, see \cite{NiSpBook,BG2006}.
Recall that a sequence $\{a_n\}_{n=1}^\infty$ of real numbers is said to be \emph{conditionally positive definite} if the infinite matrix $\{a_{i+j}\}_{i,j=1}^{\infty}$ is positive definite, see \cite{NiSpBook}.
It is equivalent to positive definiteness of the sequence $\{a_{n+2}\}_{n=0}^\infty$.

\begin{prop}[\cite{HTS}]\label{sd_vs_cumulants}
Let $\mu$ be a Borel probability measure on $\R$ with moments of all
orders, and let $\{ r_{n}(\mu)\}_{n=1}^{\infty}$ be the free cumulant
sequence of $\mu$. Then:

\begin{enumerate}[{\rm (i)}]

\item \label{BB}
	If $\mu$ is freely self-decomposable then $\{n r_{n}(\mu)\}_{n=1}^{\infty}$ is
        conditionally positive definite.

\item \label{aaa}
	Suppose further that $\mu$ has compact support. Then $\mu$
        is freely self-decomposable if and only if $\{n r_{n}(\mu)\}_{n=1}^{\infty}$
        is conditionally positive definite.

\end{enumerate}
\end{prop}


\begin{rem}\label{PD}
Suppose that $\mu$ is compactly supported. It is well known that $\mu$ is freely infinitely divisible if and only if $\{ r_n(\mu)\}_{n=1}^\infty$ is conditionally positive definite (see e.g.\ \cite[Theorem 13.16]{NiSpBook}).
Observe the following implication:
\[
\text{$\{nr_{n}(\mu)\}_{n=1}^{\infty}$ is conditionally positive definite} \quad \Longrightarrow \quad \text{$\{ r_{n}(\mu)\}_{n=1}^{\infty}$ is conditionally positive definite}.
\]
Indeed, the sequence $\{\frac{1}{n}\}_{n=1}^\infty$ is conditionally positive definite since
$\frac{1}{n}$ is the $(n-1)$-th moment of the uniform distribution on $(0,1)$ and the pointwise product of two conditionally positive definite sequences is again conditionally positive definite.
\end{rem}


\subsection{Free regular distributions}
A freely infinitely divisible distribution $\mu$ on $[0,\infty)$ is said to be \emph{free regular} if the measure $\mu^{\boxplus t}$ is also a probability measure on $[0,\infty)$ for all $t>0$. For example, the Marchenko-Pastur law $\Pi_{p,\theta}$ is free regular, where
\begin{align*}
\Pi_{p,\theta}(dx):=\max\{1-p,0\}\delta_0+\frac{\sqrt{(\theta(1+\sqrt{p})^2-x)(x-\theta(1-\sqrt{p})^2)}}{2\pi \theta x}\mathbf{1}_{(\theta(1-\sqrt{p})^2,\theta(1+\sqrt{p})^2)}(x)dx,
\end{align*}
for $p,\theta>0$ since $\Pi_{p,\theta}^{\boxplus t}=\Pi_{pt,\theta}\in\calP([0,\infty))$ for all $t>0$. In \cite{AHS13} a characterization of free regular measures is given via R-transform as follows:

\begin{thm}\label{thm:AHS13}\cite[Theorem 4.2]{AHS13}
Let $\mu$ be a freely infinitely divisible distribution on $[0,\infty)$. Then
$\mu$ is free regular if and only if its free cumulant transform is represented as
\begin{align*}
R_\mu(z)= \eta' z+ \int_\mathbb{R} \left( \frac{1}{1-zx}-1 \right) \nu(dx), \qquad(z\in\C^-),
\end{align*}
for some $\eta'\geq 0$ and $\nu$ is the free L\'{e}vy measure with $\int_{(0,\infty)} \min\{1,x^2\} \nu(dx)<\infty$ and $\nu((-\infty,0])=0$.
\end{thm}
Futhermore, free regular measures are characterized as free subordinators, see Theorem 4.2 in \cite{AHS13}.


\section{Fuss-Catalan distributions $\mu(p,p)$}
In this section we discuss the Fuss-Catalan distributions $\mu(p,p)$, $p\geq1$.
First we study free infinite divisibility. In Section 3.2 we obtain a result for free self-decomposability of $\mu(p,p)$. Then we provide free L\'{e}vy-Khintchine representation of $\mu(p,p)$ via the Gauss hypergeometric functions. In Section 3.4 we introduce concept of the free $L_1$ class and prove
that $\mu(p,p)$ is in the free $L_1$ class for all $1\leq p \leq 2$.
In Section 3.5 we investigate free regularity for $\mu(p,p)$.
Finally we prove that all the distributions $\mu(p,p)$, $p\geq1$, are unimodal.

\subsection{Free infinite divisibility for $\mu(p,p)$}
By (\ref{freecumulants of mu(p,r)}) the free cumulants of $\mu(p,p)$ are given by
$r_n(\mu(p,p))=A_n(0,p)=\binom{p}{n}$, $n\geq1$.
Therefore first we are going to study the sequence $\left\{ \binom{p}{n+2}\right\}_{n=0}^\infty$.

\begin{prop}\label{prop:binom p,n+2}
If $-1<p<2$, $p\neq 0,1$, then the sequence $\left\{\binom{p}{n+2}\right\}_{n=0}^{\infty}$
admits the following integral representation:
\begin{equation*}
\binom{p}{n+2}=\int_{-1}^{0}x^n\cdot
\frac{x\sin(p\pi)}{\pi}\left(\frac{1+x}{-x}\right)^p\,dx.
\end{equation*}
This sequence is positive definite if and only if $p\in[-1,0]\cup[1,2]$.
\end{prop}
\begin{proof}
Substituting $x\to -y$, using the properties of the Beta function
and applying Euler's reflection formula:
\begin{align*}
\Gamma(1-p)\Gamma(p)=\frac{\pi}{\sin(p\pi)}, \qquad p\notin\mathbb{Z},
\end{align*}
we get
\begin{align*}
\int_{-1}^{0}x^n\cdot
\frac{x\sin(p\pi)}{\pi}&\left(\frac{1+x}{-x}\right)^p\,dx\\
&=(-1)^{n+1}\int_{0}^{1}y^{n+1-p}(1-y)^{p}\cdot
\frac{\sin(p\pi)}{\pi}\,dy\\
&=(-1)^{n+1}\frac{\sin(p\pi)}{\pi}
\frac{\Gamma(n+2-p)\Gamma(p+1)}{\Gamma(n+3)}\\
&=(-1)^{n+1}\frac{\sin(p\pi)}{\pi}\frac{(n+1-p)\ldots(n+1-p)(2-p)(1-p)\Gamma(1-p)\times p\Gamma(p)}{(n+2)!}\\
&=\frac{p(p-1)(p-2)\ldots(p-n-1)}{(n+2)!}=\binom{p}{n+2}.
\end{align*}
If $p\in(-1,0)\cup(1,2)$, $-1<x<0$ then $x\sin(p\pi)>0$.
Therefore the sequence $a_{n}(p):=\binom{p}{n+2}$ is positive definite
for $p\in(-1,0)\cup(1,2)$. We have also $a_{n}(-1)=(-1)^n$,
$a_{n}(0)=a_{n}(1)=0$ for $n\geq0$, $a_{0}(2)=1$ and $a_{n}(2)=0$ for $n\geq1$,
so the sequences $a_{n}(-1)$, $a_{n}(0)$, $a_{n}(1)$, $a_{n}(2)$ are positive definite too. On the other hand, if the sequence $a_n(p)$ is positive definite, then $a_0(p)=p(p-1)/2\geq0$ and \[a_{0}(p)a_{2}(p)-a_{1}(p)^2=\frac{1}{144}(2-p)(1+p)p^2(p-1)^2\geq0,\]
which implies that $p\in[-1,0]\cup[1,2]$.
\end{proof}

According to Remark~\ref{PD} and Proposition~\ref{prop:binom p,n+2}, we have

\begin{thm}\label{thm:muppinfdiv}
The Fuss-Catalan distribution $\mu(p,p)$ is freely infinitely divisible if and only if $1\leq p \leq 2$.
\end{thm}

Note that this case was overlooked in Corollary~7.1 in \cite{MP2014}.

\subsection{Free self-decomposability for $\mu(p,p)$}
In this section we will prove free self-decomposability for $\mu(p,p)$, $1\leq p \leq 2$.
We need the following

\begin{prop}\label{prop:binom n+2,p,n+2}
If $0<p<2$, $p\ne1$, then  the sequence $\left\{(n+2)\binom{p}{n+2}\right\}_{n=0}^{\infty}$
admits the following integral representation:
\begin{equation}\label{integralsequenceb}
(n+2)\binom{p}{n+2}=-\int_{-1}^{0}x^n\cdot
\frac{p \sin(p\pi)}{\pi}\left(\frac{1+x}{-x}\right)^{p-1}\,dx.
\end{equation}
This sequence is positive definite if and only if $p\in\{0\}\cup [1,2]$.
\end{prop}

\begin{proof}
Similarly as before, we have
\begin{align*}
-\int_{-1}^{0} x^n&\cdot\frac{p\sin(p\pi)}{\pi}\left(\frac{1+x}{-x}\right)^{p-1}\,dx\\
&=(-1)^{n+1}\frac{p\sin(p\pi)}{\pi} \int_{0}^{1}y^{n+1-p}(1-y)^{p-1}\,dy\\
&=(-1)^{n+1}\frac{p \sin(p\pi)}{\pi}\frac{\Gamma(n+2-p)\Gamma(p)}{\Gamma(n+2)}\\
&=(-1)^{n+1}\frac{p \sin(p\pi)}{\pi}\frac{(n+1-p)(n-p)\ldots(1-p)\Gamma(1-p)\Gamma(p)}{(n+1)!}\\
&=\frac{p(p-1)\ldots(p-n-1)}{(n+1)!}
=(n+2)\binom{p}{n+2}.
\end{align*}
From (\ref{integralsequenceb}) we see that the sequence $b_{n}(p):=(n+2)\binom{p}{n+2}$ is positive definite
for $1<p<2$. We have also $b_{n}(0)=b_{n}(1)=0$ for $n\geq0$, $b_{0}(2)=2$ and $b_{n}(2)=0$ for $n\geq1$, so that the sequences $b_n(0)$, $b_n(1)$ and $b_n(2)$ are positive definite too. On the other hand, if the sequence $b_n(p)$ is positive definite, then $b_{0}(p)=p(p-1)\geq0$ and
\[
b_{0}(p)b_{2}(p)-b_{1}(p)^2=\frac{1}{12}p^3(p-1)^2(2-p)\geq0,
\]
which implies that $p\in \{0\}\cup [1,2]$.
\end{proof}

Combining Proposition~\ref{sd_vs_cumulants} with Proposition~\ref{prop:binom n+2,p,n+2} we obtain

\begin{thm}\label{thm:selfdecpp}
The Fuss-Catalan distribution $\mu(p,p)$ is freely self-decomposable if and only if $1\leq p \leq 2$.
\end{thm}

\subsection{Free cumulant transform and nonregularity of $\mu(p,p)$}
From the binomial expantion, the free cumulant transform $R_{\mu(p,p)}$ of $\mu(p,p)$ is given by
\begin{equation*}
R_{\mu(p,p)}(z)=\sum_{n=1}^\infty A_n(0,p)z^n=(1+z)^{p}-1.
\end{equation*}
Since $\mu(p,p)$ is freely infinitely divisible (even freely self-decomposable) for $1\leq p\leq2$, its free cumulant transform should be written by the free L\'{e}vy-Khintchine representation \eqref{eqno1}.
We give a free characteristic triplet of $\mu(p,p)$ for $1\leq p \leq2$. In particular $R_{\mu(1,1)}(z)=z$ and $R_{\mu(2,2)}(z)=2z+z^2$,
so in these cases the free L\'{e}vy-Khintchine triples are $(0,1,0)$ and $(1,2,0)$.

Now we assume that $1< p< 2$.
We will apply the \emph{Gauss hypergeometric series} which is defined by 
\begin{equation*}
_2F_1(a,b,c;z):=\sum_{n=0}^\infty \frac{(a)_n(b)_n}{(c)_n}\frac{z^n}{n!}, 
\end{equation*}
where $a,b,c$ are real parameters, $c\neq 0,-1,-2,\dots$
and $(a)_n$ denotes the \textit{Pochhammer symbol}: $(a)_n:=a(a+1)\dots (a+n-1)$, $(a)_0:=1$, see \cite{AAR99,OLBC}. The series is absolutely convergent for $|z|<1$.
The function $_2F_1(a,b,c;z)$ is the unique solution $f(z)$ which is analytic at $z=0$ with $f(0)=1$, of the following equation:
\begin{align}\label{eq:diffeq}
z(1-z)f''(z)+\left( c-(a+b+1)z \right) f'(z)-abf(z)=0.
\end{align}
Moreover, if $\Re(c)>\Re(b)>0$ and $|z|<1$ then we have the following integral representation:
\begin{equation}\label{eq:integral rep}
_2F_1(a,b,c;z)=\frac{1}{B(b,c-b)}\int_0^1 x^{b-1}(1-x)^{c-b-1} (1-zx)^{-a}dx,
\end{equation}
where $B(t,s)$ is the Beta function. We give an explicit formula for a special choice of $a,b,c$.

\begin{lem}\label{lem:Gaussformula} For $s>0$, $s\neq 1,2$ we have
\begin{equation}\label{2F1formula}
_2F_1(1,s,3;z)=-\frac{2\left((s-2)z+1-(1-z)^{2-s}\right)}{(s-2)(s-1)z^2}.
\end{equation}
\end{lem}
\begin{proof}

Let $s>0$, $s\neq 1,2$ and denote by $f_s(z)$ the right hand side of \eqref{2F1formula}. This function is analytic at $z=0$ and it is easy to check that $f_s(0)=1$. Then we have
\begin{align*}
f_s'(z)&=\frac{2\left( 2(1-z)^{2-s}+(2-s)z(1-z)^{1-s}- 3(2-s)z  +2\right)}{(s-2)(s-1)z^3},\\
f_s''(z)&=\frac{-2\left(6(1-z)^{2-s}+4(2-s)z(1-z)^{1-s}+(2-s)(1-s)z^2(1-z)^{-s}-6(2-s)z+6\right)}{(s-2)(s-1)z^4},
\end{align*}
and one can check that the differential equation \eqref{eq:diffeq} satisfied.
This concludes the proof.
\end{proof}

For $1<p<2$, we get the free L\'{e}vy-Khintchine representation of $\mu(p,p)$.

\begin{prop}\label{prop:LK p,p}
For $1<p<2$,  $z\in\mathbb{C}^-$, we have
\begin{equation}\label{R-pp}
R_{\mu(p,p)}(z)=pz+\int_\mathbb{R}\left( \frac{1}{1-zx}-1-zx\mathbf{1}_{[-1,1]}(x)\right)\times \left(-\frac{\sin(p\pi)}{\pi|x|}\right) \left(\frac{1+x}{-x}\right)^p \mathbf{1}_{(-1,0)}(x)dx,
\end{equation}
\end{prop}
\begin{proof}
For $1<p<2$ and for $z$ in a disc around $0$, by applying \eqref{eq:integral rep} and Lemma~\ref{lem:Gaussformula}, we have
\begin{align*}
\int_\mathbb{R}&\left( \frac{1}{1-zx}-1-zx\mathbf{1}_{[-1,1]}(x)\right)\times \left(-\frac{\sin(p\pi)}{\pi|x|}\right) \left(\frac{1+x}{-x}\right)^p \mathbf{1}_{(-1,0)}(x)dx\\
&=z^2\left(-\frac{\sin(p\pi)}{\pi}\right) \int_0^1 x^{1-p}(1-x)^p(1+zx)^{-1}dx\\
&=z^2\left(-\frac{\sin(p\pi)}{\pi}\right) B(2-p,1+p) _2F_1(1,2-p,3;-z)\\
&=z^2\left(-\frac{\sin(p\pi)}{\pi}\right) \left(-\frac{\pi(p-1)p}{2\sin(p\pi)}\right)\left(\frac{2((1+z)^p-pz-1)}{(p-1)pz^2}\right)\\
&=(1+z)^p-pz-1\\
&=R_{\mu(p,p)}(z)-pz.
\end{align*}
Therefore the free cumulant transform of $\mu(p,p)$ has the representation \eqref{R-pp} on a neighbourhood $0$ for $1<p<2$. Since $\mu(p,p)$ is freely infinitely divisible for all $1<p < 2$, its free cumulant transform has an analytic continuation to $\mathbb{C}^-$ and therefore the formula $\eqref{R-pp}$ holds for all $z\in \mathbb{C}^-$ by using the identity theorem of complex analytic functions.
\end{proof}

\begin{rem}\label{rem:LK p,p}
By Proposition~\ref{prop:LK p,p}, the free L\'{e}vy measure $\nu_{\mu(p,p)}$ of $\mu(p,p)$ is of the form
\begin{align*}\label{freeLevy1}
\nu_{\mu(p,p)}(dx)=\frac{k_p(x)}{|x|}dx,
\end{align*}
where
\begin{align*}
k_p(x)=\left(-\frac{\sin(p\pi)}{\pi}\right) \left(\frac{1+x}{-x}\right)^p \mathbf{1}_{(-1,0)}(x)dx.
\end{align*}
Note that for $1<p<2$ the function $k_p$ is non-decreasing on $(-1,0)$, which is an another proof of free self-decomposability of $\mu(p,p)$.
Moreover, the free L{\'e}vy mesure $\nu_{\mu(p,p)}$ for $1< p<2$ has compact support so that it is not
in free counterpart of so-called Thorin class.
Note also that, surprisingly, the support of $\mu(p,p)$ is contained in the positive half-line, while the support of the L{\'e}vy measure $\nu_{\mu(p,p)}$ is contained in negative half-line.
\end{rem}

As an immediate consequence of Proposition \ref{prop:LK p,p} and Theorem~\ref{thm:AHS13} we get

\begin{thm}\label{thm:FR p,p}
If $1<p<2$ then the Fuss-Catalan distribution $\mu(p,p)$ is not free regular.
\end{thm}

Since $\mu(1,1)=\delta_1$ and $\mu(2,2)^{\boxplus t}$ corresponds to the Wigner's semicircle law with mean $2t$ and variance $t$ for all $t>0$, then the measure $\mu(p,p)$ is free regular for $p=1,2$.
\subsection{$\mu(p,p)$ is in the class free $L_1$}

Now we define the following special subclass of all freely self-decomposable distributions.

\begin{defn}
A probability measure $\mu$ on $\mathbb{R}$ is said to be in the class free $L_1$
if $\mu$ is freely self-decomposable and for every $c\in (0,1)$ the measure $\rho_c:=\rho_c(\mu)\in\calP(\mathbb{R})$
in (\ref{self-decomposabledefinition}) is also freely self-decomposable.
\end{defn}

According to Section 3.1, the measure $\mu(p,p)$ is freely self-decomposable for $1\leq p \leq 2$. Therefore for any $c\in (0,1)$ there exists $\rho_{p,c}\in\calP(\mathbb{R})$ such that
\begin{equation}\label{rho_p,c}
\mu(p,p)=D_c(\mu(p,p))\boxplus\rho_{p,c}.
\end{equation}
Now we will prove that $\rho_{p,c}$ is also freely self-decomposable for $1\leq p \leq 2$.

\begin{thm}\label{thm:free L1 (p,p)}
The Fuss-Catalan distribution $\mu(p,p)$ is in the class free $L_1$ for all $1\leq p \leq 2$.
\end{thm}

\begin{proof}
For any $p\in[1,2]$, $c\in(0,1)$, we show that $\rho_{p,c}$ in \eqref{rho_p,c}  is freely self-decomposable.
First we consider the cases $p=1,2$. If $p=1$ then
\begin{equation*}
R_{\rho_{1,c}}(z)=R_{\mu(1,1)}(z)-R_{D_c(\mu(1,1))}(z)=z-cz=(1-c)z,
\end{equation*}
for all $z\in\mathbb{C}^-$. Therefore $\rho_{1,c}$ is freely self-decomposable. When $p=2$ we have
\begin{equation*}
R_{\rho_{2,c}}(z)=R_{\mu(2,2)}(z)-R_{D_c(\mu(2,2))}(z)=2(1-c)z+(1-c^2)z^2,
\end{equation*}
for all $z\in\mathbb{C}^-$. Therefore $\rho_{2,c}$ is also freely self-decomposable.

Now assume that $p\in(1,2)$. By \eqref{R-pp}, we have that the free L\'{e}vy measure $\nu_{\rho_{p,c}}$ of $\rho_{p,c}$ is given by
\begin{equation*}
-\frac{\sin(p\pi)}{\pi|x|} \left\{ \left(\frac{1+x}{-x}\right)^p \mathbf{1}_{(-1,0)}(x)- \left(\frac{c+x}{-x}\right)^p\mathbf{1}_{(-c,0)}\right\}dx.
\end{equation*}
Define
\begin{equation*}
k_{p,c}(x):=-\frac{\sin(p\pi)}{\pi}  \left\{ \left(\frac{1+x}{-x}\right)^p \mathbf{1}_{(-1,0)}(x)- \left(\frac{c+x}{-x}\right)^p\mathbf{1}_{(-c,0)}\right\},
\end{equation*}
that is, $\nu_{\rho_{p,c}}(dx)=\frac{k_{p,c}(x)}{|x|}dx$.
It is enough to show that $k_{p,c}(x)$ is non-decreasing on $(-1,0)$.
Put
\[
u_{p,c}(x):=-\frac{\pi}{\sin(p\pi)}k_{p,c}(x).
\]
Then
\begin{equation*}
u_{p,c}'(x)=\frac{p(1+x)^{p-1}}{(-x)^{p+1}}>0
\end{equation*}
for $-1<x \leq -c$ and
\begin{equation*}
u_{p,c}'(x)=\frac{p\{(1+x)^{p-1}-c(c+x)^{p-1}\}}{(-x)^{p+1}}>0
\end{equation*}
for $-c<x<0$.
Hence $u_{p,c}(x)$ is non-decreasing on $(-1,0)$, and therefore so is $k_{p,c}(x)$.
Thus $\rho_{p,c}$ is freely self-decomposable for $p\in(1,2)$ too by \cite{BNT06} or Section 2.2.1.
\end{proof}

\subsection{Unimodality for $\mu(p,p)$}

Unimodality is a remarkable property of classical and freely self-decomposable distributions, namely, every classical or freely self-decomposable distribution is unimodal (see \cite{HT,Yam}). Since $\mu(p,p)$, $p>0$, is absolutely continuous with respect to Lebesgue measure, we consider only
unimodality of the density function $W_{p,p}(x)$. According to Proposition~\ref{Prop:density} we have
\begin{align}\label{pp}
&W_{p,p}(\rho(\varphi)) =
\frac{(\csc((p-1)\varphi))\sin(\varphi)\sin(r \varphi)}{\pi} , \quad 0<\varphi<\frac{\pi}{p}.&
\end{align}
\begin{prop}\label{prop:unimodal p,p}
The Fuss-Catalan distribution $\mu(p,p)$ is unimodal for all $p\geq 1$.
\end{prop}
\begin{proof}
If $1\leq p\leq2$ then $\mu(p,p)$ is freely self-decomposable, therefore, by \cite[Theorem 1]{HT}, it is unimodal.
Now assume that $p>2$. Denote by $g(\varphi)$ the right hand side of (\ref{pp}). Then
\begin{align*}
\frac{dg(\varphi)}{d\varphi} = \frac{\left(\sin ^2(p \varphi )-p\sin ^2(\varphi )\right) \csc^2(\varphi(1 -p))}{\pi },\quad 0<\varphi <\frac{\pi}{p}.
\end{align*}
Here we consider $h(\varphi) = \sin(p \varphi ) - \sqrt{p}\sin(\varphi )$.
$h'(\varphi) =p \cos(p \varphi ) - \sqrt{p}\cos(\varphi )$ and $h''(\varphi) = - p^{2} \sin(p \varphi ) + \sqrt{p}\sin(\varphi )$.
$h'(0) =p - \sqrt{p}>0$, $h'(\frac{\pi}{p})=-p + \sqrt{p}\cos(\frac{\pi}{p})<0 $ and $h''(\varphi) <0$ on $0<\varphi <\frac{\pi}{p}$
because $\sin(p\varphi)\geq\sin(\varphi)$ on $0<\varphi <\frac{\pi}{p}$.
So this means that $\sin(p \varphi ) = \sqrt{p}\sin(\varphi )$ has only one solution in $0<\varphi <\frac{\pi}{p}$ for $p>2$
and so does $\frac{dg(\varphi)}{d\varphi} $.
As a result we obtain that $\mu(p,p)$ is unimodal.
\end{proof}


\section{General Fuss-Catalan distributions $\mu(p,r)$}
In this section, we discuss the general Fuss-Catalan disribution $\mu(p,r)$ and give a proof of Theorem \ref{thm:main}. In Section 4.1, we study free infinite divisibility for $\mu(p,r)$. In Section 4.2, we obtain a L\'{e}vy-Khintchine representation of $\mu(p,r)$ for $r<p$. In Section 4.3, we investigate free self-decomposability for $\mu(p,r)$ via its L\'{e}vy measure. In Sections 4.4, we discuss free regularity of $\mu(p,r)$. In Section 4.5, we study unimodality for four special families Fuss-Catalan distributions.

\subsection{Free infinite divisibility for $\mu(p,r)$}
The characterization of those $\mu(p,r)$ which are freely infinitely divisible in \cite[Corollary 7.1]{MP2014}, was not quite correct
and overlooked the distributions $\mu(p,p)$, $1\leq p\leq 2$. Here we provide
complete description and proof.

\begin{thm}\label{thm:FID p,r}
Suppose that $p\geq1$, $0<r\leq p$. The Fuss-Catalan distribution $\mu(p,r)$ is freely infinitely divisible if and only if either $0<r\leq \min\{p/2, p-1\}$ or $1\leq p=r \leq 2$.
\end{thm}

\begin{proof}
From Corollary~5.1 in \cite{Mlot} and Theorem~\ref{thm:muppinfdiv}
we know that if either $0<r\leq \min\{p/2, p-1\}$ or $1\leq p=r\leq2$ then $\mu(p,r)$ is freely infinitely divisible.

On the other hand, by Theorem~\ref{thm:muppinfdiv}, $\mu(p,p)$ is not freely infinitely divisible for $p>2$.
Similarly, for $\mu(p,p-1)$ the free cumulants are $A_n(1,p-1)=\binom{n-2+p}{n}$.
Since
\[
A_2(1,p-1)A_4(1,p-1)-A_3(1,p-1)^2=\frac{1}{144}p^2(p-1)^2(p+1)(2-p),
\]
$\mu(p,p-1)$ is not freely infinitely divisible for $p>2$.

Now assume that $p-1<r<p$.
We are going to show that there exists even $n\geq2$ such that $A_{n}(p-r,r)\leq 0$.
We have
\begin{align*}
r_n(p,r)&=A_{n}(p-r,r)=\frac{r}{n(p-r)+r}\binom{n(p-r)+r}{n}\\
&=\frac{r}{n!}\prod_{i=1}^{n-1}(n(p-r)+r-i)
=\frac{r}{n!}\prod_{i=1}^{n-1}(n(p-r-1)+r+i).
\end{align*}
Putting $u:=r+1-p$ we have $0<u<1$ and $A_{n}(p-r,r)\leq0$ if and only if
$nu-r>1$ and the floor $\lfloor nu-r\rfloor$ is odd. Consider the set
\[
G:=\{(nu-r)\!\!\!\mod 2:n\in\mathbb{N},\hbox{ $n$ even, }nu>r+1\}\subseteq[0,2).
\]
If $u$ is rational then $G$ is a coset of a finite subgroup
of the group $\big([0,2),0,+_{\!\!\!\!\mod2}\big)$,
otherwise $G$ is a dense subset of $[0,2)$.
It is easy to see in the former case that for some
even $n$ we have $(nu-r)\!\!\!\mod2\in[1,2)$, which implies $A_{n}(p-r,r)\leq0$.
In the latter we can find $n$ such that $n$ is even, $1<(nu-r)\!\!\!\mod2<2$
and then $A_{n}(p-r,r)<0$.

Finally, assume that $p>2$ and $p-1<r<p/2$. Then $p-r>1$ and the free cumulants $A_n(p-r,r)$
admit integral representation (\ref{eg:akprintegral}), with $p-r$ instead of $p$.
Since $r>p-r$, $W_{p-r,r}(x)$ is negative on some interval, and so is $x^2 W_{p-r,r}(x)$.
This implies, that the sequence $A_{n+2}(p-r,r)$ is not positive definite (c.f. Lemma~7.1 in \cite{MP2014}).
\end{proof}

\subsection{Free cumulant transform of $\mu(p,r)$}
In this section, we give the free L\'{e}vy-Khintchine representation of $\mu(p,r)$ to discuss free self-decomposability and free regularity for $\mu(p,r)$.

\begin{prop}\label{prop:LK p,r}
If $0<r\leq \min\{p/2, p-1\}$ then the free cumulant transform $R_{\mu(p,r)}$ has an analytic continuation to $\mathbb{C}^-$
and we have
\begin{equation}\label{eq:LK p,r}
R_{\mu(p,r)}(z)=\int_0^{(p-r)^{p-r}(p-r-1)^{1-(p-r)}}\left(\frac{1}{1-zx}-1\right) W_{p-r,r}(x)dx
\end{equation}
for all $z\in\mathbb{C}^-$, where $W_{p-r,r}(t)$ is the density function of $\mu(p-r,r)$ with respect to Lebesgue measure on $\mathbb{R}$.
\end{prop}
\begin{proof}
Suppose that $p\geq 1$ and $0<r< p$. Then the function $W_{p-r,r}(t)$ is the probability density function of the probability measure $\mu(p-r,r)$. Since $W_{p-r,r}(t)$ has a support concentrated on $[0,(p-r)^{p-r}(p-r-1)^{1-(p-r)}]$, we can take $\epsilon>0$ such that $|z|<\epsilon$ and $|zt|<1$ for all $t$ in the support of $W_{p-r,r}(t)$. If it is necessary, then we replace $\epsilon>0$ such that the free cumulant transform $R_{\mu(p,r)}(z)$ extends the power series at $z=0$. Then for $|z|<\epsilon$, the free cumulant transform $R_{\mu(p,r)}$ is written by
\begin{align*}\label{eq:LK p,r}
R_{\mu(p,r)}(z)&=\sum_{n=1}^\infty A_{n}(p-r,r)z^n=\sum_{n=1}^\infty \left(\int_0^{(p-r)^{p-r}(p-r-1)^{1-(p-r)}} x^n W_{p-r,r}(x)dx \right)z^n\\
&=\int_0^{(p-r)^{p-r}(p-r-1)^{1-(p-r)}} \sum_{n=1}^\infty (zx)^n W_{p-r,r}(x)dx\\
&=\int_0^{(p-r)^{p-r}(p-r-1)^{1-(p-r)}} \frac{zx}{1-zx} W_{p-r,r}(x)dx\\
&=\int_0^{(p-r)^{p-r}(p-r-1)^{1-(p-r)}} \left(\frac{1}{1-zx}-1\right) W_{p-r,r}(x)dx.
\end{align*}
In particular, if $0<r\leq \min\{p/2, p-1\}$, then $\mu(p,r)$ is freely infinitely divisible by Theorem \ref{thm:FID p,r}. Hence the free cumulant transform $R_{\mu(p,r)}$ has an analytic continuation to $\mathbb{C}^-$ and therefore the formula \eqref{eq:LK p,r} holds for all $z\in \mathbb{C}^-$ by using the identity theorem of complex analytic functions.
\end{proof}

From the above proposition, the free L\'{e}vy measure $\nu_{\mu(p,r)}$ of $\mu(p,r)$ is given by
\begin{equation}\label{Levy pr}
\nu_{\mu(p,r)}(dx)=\frac{k_{p,r}(x)}{x}dx, \qquad k_{p,r}(x):=x W_{p-r,r}(x),
\end{equation}
for all $p\geq 1$ and $0<r\leq \min\{p/2,p-1\}$.

\subsection{Free self-decomposability for $\mu(p,r)$}
In order to check free self-decomposability of $\mu(p,r)$ we should check
whether or not $xW_{p-r,r}(x)dx$ is unimodal with mode $0$.

\begin{thm}
Suppose that $p\geq 1$, $0<r\leq p$. The Fuss-Catalan distribution $\mu(p,r)$ is freely self-decomposable if and only if $1\leq p=r \leq 2$.
\end{thm}
\begin{proof}
In view of Theorem~\ref{thm:selfdecpp} and Theorem~\ref{thm:FID p,r}
it suffices to check the case $0<r\leq \min\{p/2,p-1\}$. By \cite[Corollary 2.5]{FL2015},  we have
\begin{align*}
k_{p,r}(x)\sim\frac{1}{\pi}\sin \left(\frac{r\pi}{p-r}\right)x^{\frac{r}{p-r}},
\end{align*}
as $x\rightarrow0^+$, where $k_{p,r}$ was defined in \eqref{Levy pr}. Since $0<r\leq \min\{p/2, p-1\}$, we have $p-r \geq r$. Hence $k_{p,r}'(x)\geq 0$ for  $x\in(0,\epsilon)$, where $\epsilon>0$ is sufficiently small. This implies that $k_{p,r}(x)dx$ can not be unimodal with mode $0$.
From remarks in Section 2.2.1, we conclude that $\mu(p,r)$ is not freely self-decomposable for $0<r\leq \min\{p/2,p-1\}$.
\end{proof}

\subsection{Free regularity for $\mu(p,r)$}
In this section, we investigate free regularlity for $\mu(p,r)$. We should check the free L\'{e}vy measure of $\mu(p,r)$.

\begin{thm}\label{thm:FR p,r}
Suppose that $p\geq1$, $0<r\leq p$. The Fuss-Catalan distribution $\mu(p,r)$ is free regular if and only if either $0<r\leq \min\{p/2, p-1\}$ or $p=r\in \{1,2\}$.
\end{thm}
\begin{proof}
We may assume that $p\geq1$, $0<r\leq p$ and either $0<r\leq \min\{p/2, p-1\}$ or $p=r$ and $1\leq p \leq 2$ holds since $\mu(p,r)$ has to be freely infinitely divisible. We have already proved that $\mu(p,r)$ is not free regular when $p=r$ and $1< p < 2$ and is free regular when $p=r\in\{1,2\}$. Suppose that $0<r\leq \min\{p/2, p-1\}$ holds. Note that $\nu_{\mu(p,r)}(dx)=W_{p-r,r}(x)dx$. Since the support of $W_{p-r,r}$ is concentrated on $[0, (p-r)^{p-r}(p-r-1)^{1-(p-r)}]$ and $\nu_{\mu(p,r)}(dx)$ is Lebesgue absolute continuous, we have $\nu_{\mu(p,r)}((-\infty,0])=0$. By Proposition \ref{prop:LK p,r} and Theorem \ref{thm:AHS13}, the distribution $\mu(p,r)$ is free regular. Therefore $\mu(p,r)$ is free regular if and only if either $0<r\leq \min\{p/2, p-1\}$ or $p=r\in \{1,2\}$.
\end{proof}

\subsection{Unimodality for $\mu(p,r)$: 4 cases}

The Fuss-Catalan distributions $\mu(p,r)$ are absolutely continuous for $p>1$, $0<r\leq p$, therefore
we have to verify unimodality of the density function $W_{p,r}(x)$.
Equivalently, it is sufficient to check whether the right hand side of (\ref{eq:propdensitywpr})
is an unimodal function for $0<\phi<\pi/p$. This turns out to be quite difficult in full generality.
We know already from Proposition~\ref{prop:unimodal p,p} that all $\mu(p,p)$, $p\geq1$, are unimodal.
Here we will study some further special cases.

\subsubsection{Case I: $\mu(p,p-1)$}

The Fuss-Catalan distributions $\mu(p,p-1)$, $p>1$, are not freely self-decomposable, but we will show that they are unimodal.

\begin{prop}\label{prop:unimodal (p,p-1)}
The Fuss-Catalan distribution $\mu(p,p-1)$ is unimodal for all $p\geq 1$.
\end{prop}
\begin{proof}
If $p=1$, then $\mu(p,p-1)=\delta_0$, and therefore it is unimodal. Assume that $p>1$.
Then the density function is
\[
W_{p,p-1}(\rho(\varphi))=\frac{1}{\pi(\cot((p-1)\varphi)+\cot\varphi)}, \qquad 0<\varphi<\frac{\pi}{p},
\]
with $\rho(\varphi)$ given by (\ref{eq:propdensityrho}).
We consider a function $g_p(\varphi)$ defined by $g_p(\varphi):=\pi W_{p,p-1}(\rho(\varphi))$ for $0<\varphi<\frac{\pi}{p}$. Then
\begin{equation*}
g_p'(\varphi)=\frac{(p-1)\sin^2\varphi+\sin^2((p-1)\varphi)}{(\cot((p-1)\varphi)+\cot\varphi)^2\sin^2((p-1)\varphi)\sin^2\varphi}.
\end{equation*}
If $p>1$, then we have that $g_p'(\varphi)>0$ for all $0<\varphi<\frac{\pi}{p}$. Hence $g_p(\varphi)$ is non-decreasing on $(0, \frac{\pi}{p})$ and therefore $\mu(p,p-1)$ is unimodal with mode $\rho(\pi/p)=p^p(p-1)^{1-p}$.
\end{proof}

\subsubsection{Case II: $\mu(2r,r)$}
We show that $\mu(2r,r)$ is unimodal for all $r\geq 1$ in this section. Note that $\mu(2r,r)$ does not have a singular part with respect to Lebesgue measure for $r\geq 1$. Therefore we consider only probability density of $\mu(2r,r)$ to study unimodality. The probability density function of $\mu(2r,r)$ is given by
\begin{align*}
W_{2r,r}(\rho(\varphi))=\frac{\sin^{r-1}((2r-1)\varphi)\sin\varphi \sin(r\varphi)}{\pi \sin^r(2r\varphi)}.
\end{align*}

\begin{prop}\label{prop;2r=p, UM}
The Fuss-Catalan distribution $\mu(2r,r)$ is unimodal for all $r\geq 1$.
\end{prop}
\begin{proof}
If $r=1$, then $\mu(2r,r)=\Pi_{1,1}$ is unimodal. Assume that $r>1$. Let $g_r(\varphi):=\pi W_{2r,r}(\rho(\varphi))$. We have
\begin{align*}
g_r'(\varphi)=\frac{G_r(\varphi)}{\sin^{r+1}(2r\varphi)},
\end{align*}
where
\begin{align*}
G_r(\varphi):&=\sin^{r-2}((2r-1)\varphi)\sin(2r\varphi) \Big[ (2r^2-3r+1)\sin\varphi \sin(r\varphi)\cos((2r-1)\varphi)\\
&+\sin((2r-1)\varphi)(r\sin\varphi\cos(r\varphi)+\cos\varphi\sin(r\varphi)) \Big]-2r\cos(2r\varphi)\sin^{r-1}((2r-1)\varphi)\sin\varphi\sin(r\varphi)
\end{align*}
Since $2r^2-3r+1>0$ and $r\varphi, (2r-1)\varphi \in (0,\pi/2)$ for all $\varphi \in (0,\frac{\pi}{2r})\subset (0,\pi/2)$ and $r>1$, we have
\begin{align*}
(2r^2-3r+1)\sin^{r-2}((2r-1)\varphi)\sin(2r\varphi) \sin\varphi \sin(r\varphi)\cos((2r-1)\varphi)\geq 0.
\end{align*}
Therefore we have
\begin{align*}
G_r(\varphi)
&\geq \sin^{r-1}((2r-1)\varphi)\sin(2r\varphi)(r\sin\varphi\cos(r\varphi)+\cos\varphi\sin(r\varphi)) -2r\cos(2r\varphi)\sin^{r-1}((2r-1)\varphi)\sin\varphi\sin(r\varphi)\\
&\geq r\sin^{r-1}((2r-1)\varphi)\sin(2r\varphi)\sin\varphi\cos(r\varphi)-2r\cos(2r\varphi)\sin^{r-1}((2r-1)\varphi)\sin\varphi\sin(r\varphi)\\
&=r\sin\varphi\sin^{r-1}((2r-1)\varphi) \Big[ \sin(2r\varphi)\cos(r\varphi)-2\cos(2r\varphi)\sin(r\varphi)\Big]\\
&=2r\sin\varphi\sin^{r-1}((2r-1)\varphi)\sin^3(r\varphi)\\
&\geq 0,
\end{align*}
for all $\varphi\in (0,\frac{\pi}{2r})$. Hence $g_r(\varphi)$ is non-decreasing on $(0,\frac{\pi}{2r})$, and therefore $\mu(2r,r)$ is unimodal for all $r>1$.
\end{proof}

\subsubsection{Case III: $\mu(1,r)$}

The density function of $\mu(1,r)$ can not be written by \eqref{eq:propdensitywpr}, but we have obtained a density formula of $\mu(1,r)$ in \cite{Mlot}.

\begin{prop}\label{prop:W_1,r}
Suppose that $0<r<1$. We have
\begin{equation}\label{eq:w1rformula}
W_{1,r}(x)=\frac{\sin(r\pi)}{\pi} x^{r-1}(1-x)^{-r}\mathbf{1}_{(0,1)}(x).
\end{equation}
Furthermore $\mu(1,r)$ is not unimodal for $0< r<1$.
\end{prop}

\begin{proof}
Equation (\ref{eq:w1rformula}) was proved in \cite[formula (5.2)]{Mlot}. Now it is elementary to check that for $0<r<1$ the function $W_{1,r}(z)$
is decreasing on $x\in(0,1-r)$ and increasing for $x\in(1-r,1)$, hence is not unimodal.
\end{proof}

Since $W_{1,p-1}(x)dx$ is the free L\'{e}vy measure of $\mu(p,p-1)$, ($1<p<2$) and it is written by \eqref{eq:w1rformula}, we get an explicit free L\'{e}vy-Khintchine representation of $\mu(p,p-1)$. From the form of free cumulants of $\mu(p,p-1)$, its free cumulant transform $R_{\mu(p,p-1)}$ is written by
\begin{equation*}
R_{\mu(p,p-1)}(z)=\sum_{n=1}^\infty A_n(1,p-1)z^n=\frac{1}{(1-z)^{p-1}}-1.
\end{equation*}
Similarly as in Proposition \ref{prop:LK p,p}, we get the following formula.

\begin{cor}\label{lem:LK p,p-1}
For $1<p<2$, we have
\begin{equation}\label{eq:LK p,p-1}
\begin{split}
R_{\mu(p,p-1)}(z)&=(p-1)z+\int_{\mathbb{R}} \left( \frac{1}{1-zx}-1-zx \mathbf{1}_{[-1,1]} (x)\right) \times \left(-\frac{\sin(p\pi)}{\pi}\right) x^{p-2}(1-x)^{1-p} \mathbf{1}_{(0,1)}(x)dx\\
&=\int_{\mathbb{R}} \left( \frac{1}{1-zx}-1\right) \times \left(-\frac{\sin(p\pi)}{\pi}\right) x^{p-2}(1-x)^{1-p} \mathbf{1}_{(0,1)}(x)dx,
\end{split}
\end{equation}
for all $z\in\mathbb{C}^-$.
\end{cor}
\begin{proof}
For all $z$ in a neighborhood of $0$, by applying the integral representation \eqref{eq:integral rep} and Lemma \ref{lem:Gaussformula} again, we have
\begin{align*}
\int_{\mathbb{R}} &\left( \frac{1}{1-zx}-1-zx \mathbf{1}_{[-1,1]} (x)\right) \times \left(-\frac{\sin(p\pi)}{\pi}\right) x^{p-2}(1-x)^{1-p} \mathbf{1}_{(0,1)}(x)dx\\
&=z^2\left(-\frac{\sin(p\pi)}{\pi}\right) \int_0^1 \frac{x^p(1-x)^{1-p}}{1-zx}dx\\
&=z^2\left(-\frac{\sin(p\pi)}{\pi}\right) B(p+1,2-p) _2F_1(1,p+1,3;z)\\
&=z^2\left(-\frac{\sin(p\pi)}{\pi}\right) \left( -\frac{\pi(p-1)p}{2\sin(p\pi)}\right) \left( -\frac{2(pz-z+1-(1-z)^{1-p})}{(p-1)pz^2}\right)\\
&=-pz+z-1+\frac{1}{(1-z)^{p-1}}\\
&=(1-p)z+R_{\mu(p,p-1)}(z).
\end{align*}
Therefore the free cumulant transform of $\mu(p,p-1)$ has the representation \eqref{eq:LK p,p-1} on the neighborhood of $0$ for $1<p<2$. Since $\mu(p,p-1)$ is freely infinitely divisible for all $1<p < 2$, its free cumulant transform has an analytic continuation to $\mathbb{C}^-$ and therefore the formula of $\eqref{eq:LK p,p-1}$ hold for all $z\in \mathbb{C}^-$ by using the identity theorem of complex analytic functions.
\end{proof}

\subsubsection{Case IV: $\mu(2,r)$}
According to Theorem \ref{thm:FID p,r}, the Fuss-Catalan distribution $\mu(2,r)$ is not freely infinitely divisible, (and therefore is not freely self-decomposable) for all $1<r<2$. However we obtain unimodality for $\mu(2,r)$ for some $r>1$. Note that $\mu(2,r)$ does not have a singular part with respect to Lebesgue measure for $1<r<2$. Therefore we consider only probability density of $\mu(2,r)$ to study unimodality. The probability density function of $\mu(2,r)$ is given by
\begin{equation*}
W_{2,r}(\rho(\varphi))=\frac{\sin(r\varphi)}{\pi\cdot 2^{2-r}(\cos\varphi)^{2-r}}, \qquad 0<\varphi<\frac{\pi}{2}.
\end{equation*}

\begin{prop}\label{prop:unimodal 2,r}
There exists $r_0\in (\frac{3}{2},2)$ such that the Fuss-Catalan distribution $\mu(2,r)$ is unimodal for all $1<r<r_0$, where $r_0$ is a unique solution in $\frac{3}{2}< r< 2$ of the equation
\begin{equation}\label{equal}
r\sin\left( \frac{3r+3}{2r+4} \pi\right)+2\sin\left(\frac{3r}{2r+4}\pi\right) \cos\left( \frac{3}{2r+4} \pi \right)=0.
\end{equation}
\end{prop}

\begin{proof}
First, we show the existance and uniqueless of a solution $r_0 \in (\frac{3}{2},2)$ of the equation \eqref{equal}. Let $A(r)$ be the function of LHS of the equation \eqref{equal}. It is easy to check that $A(r)>0$ for all $r\in (1,\frac{3}{2})$ and $A(2)<0$. Moreover $A'(r)<0$ for all $\frac{3}{2}<r<2$. Hence there exists a unique solution $r_0\in(\frac{3}{2},2)$ such that $A(r_0)=0$ by the intermediate value theorem.

Assume $1<r<r_0$. We consider a function $g_r(\varphi)$ defined by $g_r(\varphi):=\pi \cdot 2^{2-r} W_{2,r}(\rho(\varphi))$ for $0<\varphi<\frac{\pi}{2}$. Then
\begin{equation*}
g_r'(\varphi)=\frac{r\cos((r+1)\varphi)+2\sin(r\varphi)\sin\varphi}{(\cos\varphi)^{3-r}}
\end{equation*}
It is sufficient to check the positivity of the function $h_r(\varphi):=r\cos((r+1)\varphi)+2\sin(r\varphi)\sin\varphi$ to see the positivity of $g_r'(\varphi)$ since $(\cos\varphi)^{3-r}>0$ for all $0<\varphi<\frac{\pi}{2}$. Since $\sin(r\varphi)\sin(\varphi)>0$ for all $\varphi \in (0,\frac{\pi}{2})$, we have that $h_r(\varphi)>0$ for all $\varphi\in (0, \frac{\pi}{2(r+1)})$. Next we show that $h_r'(\varphi)<0$ for all $\varphi\in(\frac{\pi}{2(r+1)},\frac{3\pi}{2(r+2)})$. To show this, we divide two cases of the region of $\varphi$.\\

Case I $\frac{\pi}{2(r+1)}<\varphi\leq\frac{\pi}{r+1}$: It is clear that
\begin{equation*}
\begin{split}
h_r'(\varphi)&=-r(r+1)\sin((r+1)\varphi)+2r\cos(r\varphi)\sin\varphi+2\sin(r\varphi)\cos\varphi\\
&=-(r-1)(r\sin((r+1)\varphi)+2\sin(r\varphi)\cos\varphi)\\
&<0.
\end{split}
\end{equation*}

Case II $\frac{\pi}{r+1}<\varphi<\frac{3\pi}{2(r+2)}$: Since $A(r)>0$ for all $1<r<r_0$, we have
\begin{equation*}
\begin{split}
h_r'(\varphi)&<-(r-1) \left[ r\sin\left( \frac{3r+3}{2r+4} \pi\right)+2\sin\left(\frac{3r}{2r+4}\pi\right) \cos\left( \frac{3}{2r+4} \pi \right)\right]\\
&=-(r-1)A(r)\\
&<0.
\end{split}
\end{equation*}
Due to the above evaluation, we obtain that $h_r'(\varphi)<0$ for all $\frac{\pi}{2(r+1)}<\varphi < \frac{3\pi}{2(r+2)}$. In addition, for all $\frac{\pi}{r+1}<\varphi<\frac{\pi}{2}$, we have
\begin{equation*}
h_r''(\varphi)=-(r-1)\left[ r(r+1)\cos((r+1)\varphi)+2r\cos(r\varphi)\cos\varphi-2\sin(r\varphi)\sin\varphi\right]>0.
\end{equation*}
Moreover $h_r'(\frac{\pi}{2})=-(r-1)r\sin(\frac{(r+1)\pi}{2})>0$ for all $1<r<r_0$. Hence there exists a unique solution $\varphi_0\in (\frac{3\pi}{2(r+2)},\frac{\pi}{2})$ such that $h_r'(\varphi_0)=0$ by the intermediate value theorem. More strongly, we have that $h_r'(\varphi)<0$ for all $\varphi\in (\frac{3\pi}{2(r+2)},\varphi_0)$ and $h_r'(\varphi)>0$ for all $\varphi\in (\varphi_0,\frac{\pi}{2})$.  The equality $h_r'(\varphi_0)=0$ implies that
\begin{equation*}
\begin{split}
(\cos\varphi_0)h_r(\varphi_0)&=r\cos((r+1)\varphi_0)\cos\varphi_0+2\sin(r\varphi_0)\cos\varphi_0\sin\varphi_0\\
&= r\cos((r+1)\varphi_0)\cos\varphi_0-r\sin((r+1)\varphi_0)\sin\varphi_0\\
&=r\cos((r+2)\varphi_0)>0,
\end{split}
\end{equation*}
since $\frac{3\pi}{2}<(r+2)\varphi_0<\frac{(r+2)\pi}{2}<2\pi$. Thus $h_r(\varphi_0)>0$, and therefore $h_r(\varphi)>0$ for all $\varphi \in (\frac{\pi}{2(r+1)},\frac{\pi}{2})$. Hence $g_r(\varphi)$ is non-decreasing on $(0,\frac{\pi}{2})$. This means that $\mu(2,r)$ is unimodal for all $1<r<r_0$.
\end{proof}


\subsection{Phase transition}
The results of Section~4.5.3 and 4.5.4, as well as some numerical
experiments, suggest,
that for every $p>1$ the Fuss-Catalan distributions admit the following
phase transition:

\begin{con}
For every $p>1$ there exists $r_{0}(p)$, with $p-1<r_0(p)<p$, such that
the Fuss-Catalan distribution $\mu(p,r)$
is unimodal if and only if either $r=p$ or $0<r\leq r_0(p)$.
\end{con}
 As an example we present  on Figure~1 graphs of $W_{2,r}(x)$ for
$r=1.5,\,1.6,\,1.7,\,1.8,\,1.9,\,2$,
the left parts of the graphs appear in this order from the top to the bottom.
Note that if $0<r<2$ then $\lim_{x\to 0^-}W_{2,r}(x)=+\infty$ and
$\lim_{x\to0^-}W_{2,2}(x)=0$.
We see that $W_{2,1.5}(x),\, W_{2,1.6}(x)$ are unimodal,
$W_{2,1.7}(x),\,W_{2,1.8}(x),\,W_{2,1.9}(x)$
are not unimodal and $W_{2,2}(x)$ is again unimodal. We have found
numerically that
for $p=2$ the phase transition is at $r_0(2)=1.6756\ldots$.
\begin{figure}[htbp]
 \centering
  \includegraphics[width=100mm]{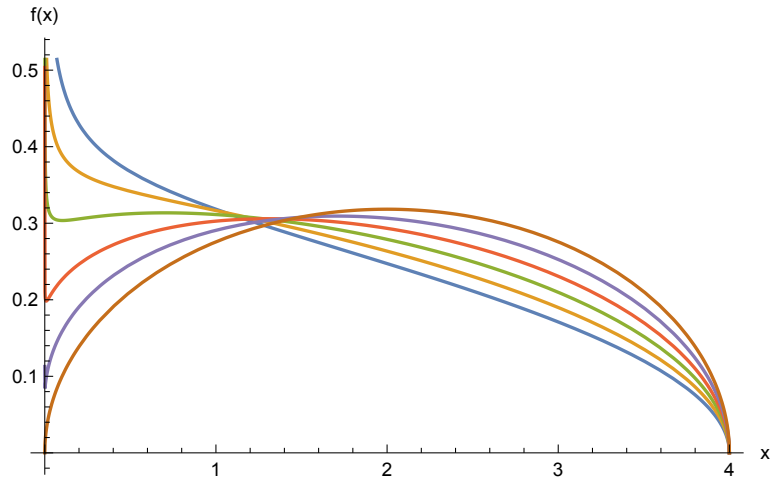}
 \caption{Density function $W_{2,r}$ for $r=1.5,\,1.6,\,1.7,\,1.8,\,1.9,\,2$}
 \label{fig:one}
\end{figure}

\newpage
On Figure~2 we represent the Fuss-Catalan distributions $\mu(p,r)$,
$p\geq1$, $0<r\leq p$, on the $(p,r)$-plane.
The middle thick blue line represents $r_0(p)$ which was found
experimentally.
Top thick red line segment corresponds to the freely self-decomposable
distributions $\mu(p,r)$ and green area corresponds free regular
infinitely divisible distributions $\mu(p,r)$. The union of the red and
green areas corresponds to the freely infinitely divisible Fuss-Catalan
distributions.

\begin{figure}[htbp]
 \centering
  \includegraphics[width=100mm]{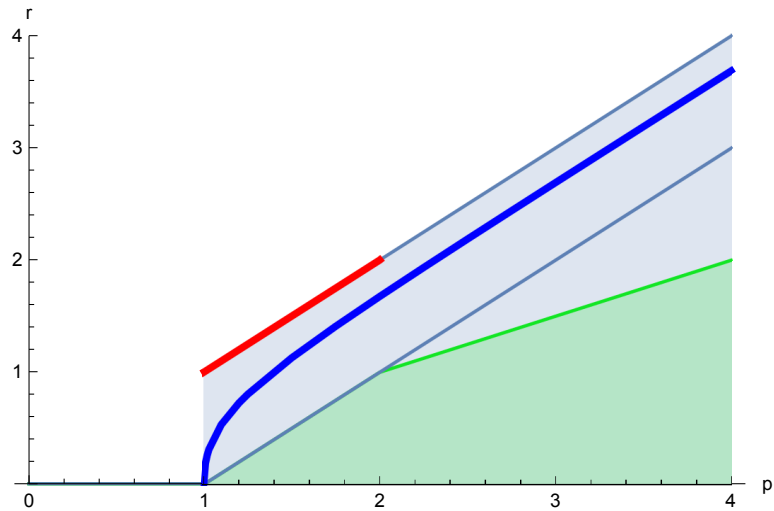}
 \caption{Phase transition}
 \label{fig:one}
\end{figure}


\end{document}